\documentclass[10pt,a4paper]{article}
\usepackage[utf8]{inputenc}
\usepackage[T1]{fontenc}
\usepackage[english]{babel}
\usepackage[bottom=3cm, top=3cm, left=2cm, right=2cm]{geometry}
\usepackage{amsmath}
\usepackage{amsfonts}
\usepackage{amssymb}
\usepackage{amsthm}
\usepackage{graphicx}
\usepackage{color}
\usepackage{mathrsfs}
\usepackage{lmodern}
\usepackage[most]{tcolorbox}
\usepackage{framed}
\usepackage{fancybox}
\usepackage{tikz}
\usepackage{bbm}
\usepackage{dsfont}
\usepackage{wrapfig}
\usepackage{hyperref}
\usepackage[pagewise]{lineno}
\allowdisplaybreaks

\newtcolorbox{boxrd}[2][]{enhanced,colback=white,width={\textwidth},
attach boxed title to top left={yshift={-0.5\baselineskip},xshift=1cm}, 
title={#2},
boxrule=0.5pt,
coltitle=black,
boxed title style={
  borderline={-0.5mm}{black}
  colframe=white,
  colback=gray!50,
  colupper={black},
},
}

\newtcolorbox{boxsp}[2][]{%
  enhanced,colback=white,colframe=black,coltitle=black,
  boxrule=0.4pt,
  fonttitle=\itshape,
  attach boxed title to top left={yshift=-0.5\baselineskip-0.3pt,xshift=2mm},
  boxed title style={tile,size=minimal,left=0.5mm,right=0.5mm,
    colback=white,before upper=\strut},
  title=#2,#1
}

\newcommand{\N}{\mathbb{N}}

\newcommand{\R}{\mathbb{R}}
\newcommand{\C}{\mathbb{C}}

\newcommand{\lessim}{\lesssim}

\newtheorem{theorem}{Theorem}
\newtheorem{proposition}{Proposition}
\newtheorem{lemma}{Lemma}

\theoremstyle{definition}
\newtheorem{remark}{Remark}

\usepackage{fancyhdr}
\usepackage{changepage}

\title{\textbf{Existence of traveling waves for the two-dimensional Zakharov system}}
\author{Guillaume Rialland}
\date{\footnotesize{Université de Paris-Saclay, UVSQ, CNRS, Laboratoire de Mathématiques de Versailles, 78000 Versailles }}

\begin{document}
\maketitle 

\begin{adjustwidth}{80pt}{80pt}
\small{\textsc{Abstract.} Via a fixed point argument, we construct solitary waves for the two-dimensional Zakharov system that travel with any small speed $c \in \R^2$. Moreover, we investigate their asymptotic behavior.}
\end{adjustwidth}

\textcolor{white}{a} \\ \\ \textcolor{white}{a}

In this paper we study the two-dimensional Zakharov system
\begin{equation}
\left \{ \begin{array}{l} \partial_t u = i \Delta u - inu \\
\partial_t n = - \nabla \cdot v \\
\partial_t v = - \nabla n - \nabla (|u|^2 ). \end{array} \right.
\label{ZL}
\end{equation}
for $(t \, , x) \in \R \times \R^2$. The associated Cauchy problem is well-posed for $(u \, , n \, , v)$ in $H^1 ( \R^2 \, , \C ) \times L^2 ( \R^2 \, , \R ) \times H^{-1} ( \R^2 \, , \R^2 )$. We refer to \cite{Bou,Gin,Sa,SuSu} for the analysis of the well-posedness of the Zakharov system in various spaces. 
The system \eqref{ZL} above is actually a convenient reformulation of the following system, first introduced by V. E. Zakharov in \cite{Za} to describe the propagation of Langmuir turbulence in plasma:
\begin{equation}
\left \{ \begin{array}{l} \partial_t u = i \Delta u - inu \\
\partial_t^2 n = \Delta ( n + |u|^2 ). \end{array} \right.
\label{Z}
\end{equation}
See also \cite{Gib} for a direct Lagrangian derivation of \eqref{ZL} and \cite{Ki} for comments on the physical interest of these systems and their solitary waves.
The equivalence between systems \eqref{ZL} and \eqref{Z} requires some caution. If $(u \, , n \, , v) \in H^1 \times L^2 \times L^2$ is a solution of \eqref{ZL} then $(u \, , n )$ is a solution of \eqref{Z} and $\partial_t n = - \nabla \cdot v \in H^{-1} ( \R^d \, , \R )$ for all time. Conversely, take $(u \, , n) \in H^1 \times L^2$ a solution of \eqref{Z}. Assume that, for any time $t$, there exists $v(t) \in L^2 ( \R^2)$ such that $\partial_t n (t) = - \nabla \cdot v(t)$. Then $(u \, , n \, , v)$ is a solution (in $H^1 \times L^2 \times L^2$) of \eqref{ZL}. It is known that, if $(u_0 \, , n_0 \, , n_1 ) \in H^1 \times L^2 \times H^{-1}$ is such that $n_1 = - \nabla \cdot v_0$, then the solution $(u(t) \, , n (t) \, , \partial_t n(t))$ of \eqref{Z} with initial data $(u_0 \, , n_0 \, , n_1)$ satisfies:
\[ \forall t , \, \, \, \exists v(t) \in L^2 ( \R^2 \, , \R^2 ), \, \, \partial_t n(t) = - \nabla \cdot v(t). \]
This useful property can be found in the introduction of \cite{Me}. In several papers (see \cite{Bou,GM,Sa}), the space of functions $w$ such that there exists $\tilde{w} \in L^2 ( \R^2 \, , \R^2 )$ such that $w = \nabla \cdot \tilde{w}$ is denoted $\hat{H}^{-1} ( \R^2 \, , \R )$. With this notation, the Cauchy problem associated with \eqref{Z} is well-posed for $(u \, , n \, , \partial_t n ) \in H^1 ( \R^d \, , \C ) \times L^2 ( \R^d \, , \R ) \times \hat{H}^{-1} ( \R^d \, , \R )$.
A solution to system \eqref{ZL} preserves the following quantities through time:
\begin{itemize}
	\item the mass $M(u ) = \displaystyle{ \int_{\R^d} |u|^2}$;
	\item the energy $H(u \, , n \, , v) = \displaystyle{\int_{\R^d} \left ( | \nabla u |^2 + n |u|^2 + \frac{n^2}{2} + \frac{|v|^2}{2} \right )}$;
	\item the momentum $P(u \, , n \, , v) = \displaystyle{\text{Im} \left (  \int_{\R^d} \overline{u} \nabla u \right ) + \int_{\R^d} nv}$.
\end{itemize}
Note that the energy and the momentum require the use of the function $v$, thus require relying on the formulation \eqref{ZL}. 
In the present paper we shall construct solitary wave solutions to the system \eqref{ZL} and investigate their regularity and asymptotic behavior. Actually, there exists a well-known standing wave solution of \eqref{ZL} for any pulsation $\omega >0$, namely 
\[ u_\omega (x) = \sqrt{\omega} \, Q( \sqrt{\omega} x), \quad n_\omega (x) = - \omega Q^2 ( \sqrt{\omega} \, x ) \quad \text{and} \quad v_\omega (x)=0, \]
where $Q \in H^1 ( \R^2 )$ is the unique positive radial ground state of $\Delta Q = Q - Q^3$, or more generally a solution of this equation. However, this standing wave does not generate travelling waves; note that, contrary to the NLS equation or wave equation, no Galilean or Lorentz transform exists for the Zakharov system. 

Recall that, in one dimension, if we denote by $Q_1$ the one-dimensional cubic NLS soliton ($Q_1''=Q_1-Q_1^3$), then
\begin{equation} 
(u \, , n \, , v)(x \, , t) = \left ( e^{i \left ( \frac{cx}{2} - \frac{c^2t}{4} + \omega t \right )} \sqrt{1-c^2} \sqrt{\omega} \, Q_1(\sqrt{\omega} (x-ct)) \, , - \omega Q_1^2 ( \sqrt{\omega} (x-ct)) \, , -c \omega Q_1^2 ( \sqrt{\omega} (x-ct)) \right ) 
\label{sol1D}
\end{equation}
is a solution to the one-dimensional Zakharov system, for any $\omega >0$ and $c \in (-1 \, , 1)$. Hence the 1D Zakharov system admits solitary waves of any velocity smaller than $1$, with comfortable explicit expressions. See \cite{Oh,Wu} for the study of these 1D solitons and their orbital stability. However, such an easy transposition from Schrödinger solitons to Zakharov solitons does not happen in two dimensions: adapting the formula \eqref{sol1D} to the two-dimensional NLS soliton does not lead to a solution to system \eqref{ZL}. 

In the present paper we construct travelling waves solutions to \eqref{ZL}, for any celerity $c$ small enough. More explicitly, we shall construct solutions $(u \, , n \, , v)$ to \eqref{ZL} of the form
\begin{equation}
\begin{split}
& u(t \, , x) = \sqrt{\omega} \, U_c \left ( \sqrt{\omega} (x-ct) \right ) e^{i \left ( \frac{c \cdot x}{2} - \frac{|c|^2t}{4} + \omega t \right )}, \\
& n(t \, , x) = \omega N_c \left ( \sqrt{\omega} (x- ct) \right ) \\
\text{and} \quad & v(t \, , x) = \omega V_c \left ( \sqrt{\omega} (x- ct) \right ).
\end{split}
\label{unv}
\end{equation}
Our main result is the following.

\begin{theorem}\label{thm1}
There exists $c_* \in (0 \, , 1)$ such that, for any $c \in \R^2$ with $|c| < c_*$, there exist $(U_c \, , N_c \, , V_c) \in H^1 ( \R^2 \, , \C ) \times L^2 ( \R^2 \, , \R ) \times L^2 ( \R^2 \, , \R^2 )$ such that, for any $\omega >0$, the functions $(u \, , n \, , v)$ defined by \eqref{unv} are solutions to the system \eqref{ZL}. The functions $U_c$, $N_c$ and $V_c$ belong to $H^s$ for any $s \geqslant 0$, and satisfy:
\[ \left \| U_c - Q \right \|_{H^2} \lesssim |c|^2, \quad \left \| N_c + Q^2 \right \|_{H^2} \lesssim |c|^2 \quad \text{and} \quad \left \| V_c \right \|_{H^2 \times H^2} \lesssim |c| . \]
Moreover, for any $m \in \N^2$, and $|y| \geqslant 1$,
\begin{equation}
|\partial^m U_c(y)| \lesssim_m e^{- \frac{1}{2} |y|} , \quad | \partial^m N_c (y) | \lesssim_m e^{-|y|} + \frac{|c|^2}{|y|^{|m|+2}} \quad \text{and} \quad | \partial^m V_{c} (y) | \lesssim \frac{|c|}{|y|^{|m|+2}} .
\label{decays}
\end{equation}
\end{theorem}

\begin{remark} \label{rk1}
We shall see that, contrary to Schrödinger and 1D Zakharov solitons, the soliton $U_c$ is not radial. More precisely, $\text{Re} (U_c)$ is radial while $\text{Im} (U_c)$ is odd in the direction of $c$ and even in the direction orthogonal to $c$. Besides, while $U_c$ has exponential decay, \eqref{decays} shows that it is not the case for $N_c$ and $V_c$. This weaker asymptotic also differs from the Schrödinger and 1D Zakharov solitons.
\end{remark}

\textbf{Organization of the paper.} We construct these solitary waves by a fixed point argument in Section \S\ref{sec1} below. We reduce our system to an equation only on the function $U$, which we then linearize around the Schrödinger soliton $Q$ in order to apply Banach's standard fixed point theorem. We then construct the functions $N$ and $V$ and check their regularity. In Section \S\ref{sec2}, we establish the asymptotic estimates \eqref{decays}: see Lemma \ref{lemAgmon} and Lemma \ref{lemDecrNV}. In Lemma \ref{lemDecrNV}, we go a little further and investigate a power expansion for $N_c$ at any order. 

\begin{remark} \label{rk2}
The construction of solitary waves via a fixed point argument is a standard method; for instance see \cite{ChiPa} for a similar construction of solitary waves for the 2D Gross-Pitaevskii equation. We also refer to \cite{GM} for the use of a fixed point argument to construct blow-up solutions to the 2D Zakharov system. Moreover, in \cite{GM} the topological issue of the unboundedness of the branch of solutions $(P_\lambda \, , N_\lambda )$ constructed is investigated, although it is not established whether the branch holds for any $\lambda >0$ or only up to a certain finite $\lambda^*$. In the present paper, we do not investigate the matter of extending the branch of solitons $(U_c \, , N_c \, , V_c)$ to non-small celerities $c$. Finally we also refer to \cite{HS} for a construction of solitary waves for fractional NLS equations that does not rely on fixed point arguments but instead on a variational method. 
\end{remark}

\begin{remark} \label{rk3}
With Theorem \ref{thm1} we hope to pave the way for the construction of multi-solitons for the 2D Zakharov system. Multi-solitons for the 1D Zakharov system have been constructed in \cite{GR3}, using the explicit expression \eqref{sol1D} for the 1D Zakharov solitary waves. In most usual constructions of multi-solitons for dispersive equations, the exponential decay of the solitons is of great use in the proof, enabling to handle the interactions between two solitons located away from one another. See \cite{Cot3,Cot4,Cot2,Gus,LeC1,LeC2,Ma5,MaMe3,GR3} for the contruction of such multi-solitons for NLS-like, gKdV-like and Zakharov equations. Here, for the 2D Zakharov system, while the component $U_c$ of the soliton decreases exponentially, the components $N_c$ and $V_c$ have far weaker decays. In this situation, it is paramount to understand well enough this non-exponential decay (which is the point of Lemma \ref{lemDecrNV}) in order to succeed in constructing multi-solitons. The construction of 2D Zakharov multi-solitons is not carried through in the present paper, but we refer to \cite{KMR} for a similar construction of multi-solitons in a case where the decay of the solitons is not exponential. 
\end{remark}

\textbf{Notation.} We denote by $Q$ the standard cubic NLS soliton: $Q$ is the unique positive radial solution in $H^1 ( \R^2 )$ of the equation $\Delta Q = Q-Q^3$. We introduce the standard linearization operators around NLS solitons:
\[ L_+ = - \Delta +1-3Q^2 \quad \text{and} \quad L_- = - \Delta + 1 - Q^2. \] 
We refer to \cite{CGNT,We2} for standard properties of $Q$ and the operators $L_\pm$. Recall for example that
\[ Q(y) \lesssim \frac{e^{-|y|}}{|y|^{1/2}} , \quad \text{Ker} (L_+) = \text{span} ( \partial_{y_1} Q \, , \partial_{y_2} Q) \quad \text{and} \quad \text{Ker} (L_-) = \text{span} (Q). \]
For $m = (m_1 \, , m_2) \in \N^2$, we denote by $\partial^m$ the differential operator $\partial_{y_1}^{m_1} \partial_{y_2}^{m_2}$. We also write $|m| = m_1+m_2$. The canonical basis of $\R^2$ will be denoted $(e_1 \, , e_2)$. The notation $\hat{f}$ denotes the Fourier transform of $f$ and $\mathcal{F}^{-1}$ denotes the inverse Fourier transform. The letter $C$ indicates a constant whose value can change from one line to another; if $C$ depends on any parameter $\mathfrak{p}$ we shall write $C_{\mathfrak{p}}$. Finally, the notation $\mathbf{A} \lesssim \mathbf{B}$ means that there exists a constant $C>0$ such that $\mathbf{A} \leqslant C \mathbf{B}$. The implicit constant $C$ does not depend on any parameter or function, unless it is explicitly indicated. The notation $\mathbf{A} \lesssim_{\mathfrak {p}} \mathbf{B}$ indicates that the implicit constant $C$ may depend on the parameter $\mathfrak{p}$.

\textbf{Acknowledgments.} This paper is the result of many discussions with Yvan Martel. May he be warmly thanked for it here. 

\section{Construction of the solitary waves} \label{sec1}

Straightforward computations show that $(u \, , n \, , v)$ satisfies \eqref{ZL} if and only if $(U \, , N \, , V)$ satisfies the following stationary system:
\begin{equation}
\left \{ \begin{array}{l} \Delta U = U + NU \\ \nabla \left ( N + |U|^2 \right ) = c \cdot \nabla V \\ c \cdot \nabla N = \nabla \cdot V \end{array} \right.
\label{sysNUV}
\end{equation}
where $c \cdot \nabla V = c_1 \partial_{y_1} V + c_2 \partial_{y_2} V$ (recall that $V$ has values in $\R^2$). We shall look for $V$ under the form $V = \nabla W$. It follows that $\nabla \cdot V = \Delta W$ and $c \cdot \nabla V = \nabla ((c \cdot \nabla ) W )$. Thus
\[ \nabla \left ( N + |U|^2 \right ) = c \cdot \nabla V \quad \Longleftrightarrow \quad \nabla \left ( N + |U|^2 \right ) = \nabla \left ( (c \cdot \nabla ) W \right ) \quad \Longleftrightarrow \quad N+|U|^2 = (c \cdot \nabla )W . \]
Inserting $N=-|U|^2 - ( c \cdot \nabla ) W$ into the third line of \eqref{sysNUV}, straightforward calculation show that
\[ c \cdot \nabla N = \nabla \cdot V \quad \Longleftrightarrow \quad \Delta_c W = - (c \cdot \nabla ) ( |U|^2 ) \]
where $\Delta_c := \Delta - (c \cdot \nabla )^2$. Hence
\[ \left \{ \begin{array}{l} V = \nabla W \\ \eqref{sysNUV} \end{array} \right. \quad \Longleftrightarrow \quad \left \{ \begin{array}{l} \Delta U = U+NU \\ V = \nabla W = - \nabla \Delta_c^{-1} (c \cdot \nabla ) (|U|^2) \\ N = -|U|^2 - \Delta_c^{-1} (c \cdot \nabla )^2 (|U|^2) \\ W = - \Delta_c^{-1} (c \cdot \nabla ) (|U|^2 ). \end{array} \right. \]
Let us define the useful operators $S_c = \Delta_c^{-1} (c \cdot \nabla )^2$, $T_{c,1} = \Delta_c^{-1} (c \cdot \nabla ) \partial_{y_1}$ and $T_{c,2} = \Delta_c^{-1} (c \cdot \nabla ) \partial_{y_2}$. Explicitly,
\begin{equation} 
S_c f = \mathcal{F}^{-1} \left [ \frac{(c \cdot \xi )^2}{| \xi |^2 - ( c \cdot \xi )^2} \, \hat{f} ( \xi ) \right ] \quad \text{and} \quad T_{c,j} f = \mathcal{F}^{-1} \left [ \frac{(c \cdot \xi ) \xi_j}{| \xi |^2 - (c \cdot \xi )^2} \, \hat{f} ( \xi ) \right ] .
\label{Sc}
\end{equation}
To sum up, if $(U \, , N \, , V)$ satisfies the system
\begin{equation} 
\left \{ \begin{array}{l} \Delta U = U+NU \\ N = -|U|^2 - S_c(|U|^2) \\ V = - \nabla \Delta_c^{-1} (c \cdot \nabla )(|U|^2) = - \left ( \begin{array}{c} T_{c,1} (|U^2|) \\ T_{c,2} (|U^2|) \end{array} \right ) \end{array} \right. 
\label{sysNUV2}
\end{equation}
then $(U \, , N \, , V)$ satisfies the system \eqref{sysNUV}. We focus on finding solutions of system \eqref{sysNUV2}. The operators $S_c$ and $T_{c,j}$ satisfy the following properties.

\begin{lemma} \label{lemSc}
Take $s \geqslant 0$, $c \in (-1 \, , 1)$. The operators $S_c$ and $T_{c,j}$ ($j=1,2$) given by \eqref{Sc} are well-defined bounded linear operators $H^s \to H^s$ and we have
\begin{equation} 
|||S_c|||_{H^s \to H^s} \leqslant \frac{|c|^2}{1-|c|^2} \quad \text{and} \quad |||T_{c,j}|||_{H^s \to H^s} \leqslant \frac{|c|}{1-|c|^2} .
\label{Sc1}
\end{equation}
Moreover, the application $c \in (-1 \, , 1) \mapsto S_c \in \mathcal{L} (H^s)$ is differentiable and, for all $f, \tilde{f} \in H^s$, $c,\tilde{c} \in (-1 \, , 1)$ and $j \in \{ 1 \, , 2 \}$,
\begin{equation} 
\left \| S_c (f) - S_{\tilde{c}} ( \tilde{f} ) \right \|_{H^s} \leqslant \frac{2|c - \tilde{c}|}{(1-|c|^2)(1-|\tilde{c}|^2)} \left \| f \right \|_{H^s} + \frac{| \tilde{c} |^2}{1- | \tilde{c} |^2} \left \| f - \tilde{f} \right \|_{H^s} 
\label{Sc2}
\end{equation}
and
\begin{equation}
\left \| \partial_{(c_j)} S_c (f) - \partial_{(c_j)} S_{\tilde{c}} ( \tilde{f} ) \right \|_{H^s} \leqslant \frac{6|c- \tilde{c}|}{(1-|c|^2)^2 (1-| \tilde{c}|^2)^2} \left \| f \right \|_{H^s} + \frac{2| \tilde{c}|}{(1-| \tilde{c} |^2)^2} \left \| f - \tilde{f} \right \|_{H^s}. 
\label{Sc3}
\end{equation}
\end{lemma}

\begin{proof}
We prove the result for $S_c$. The proof for $T_{c,j}$ is analogous. For all $\xi \in \R^2$, we have $|(c \cdot \xi )|^2 \leqslant |c|^2 | \xi |^2$ thus $\left | \frac{(c \cdot \xi )^2}{| \xi |^2 - (c \cdot \xi )^2} \right | \leqslant \frac{|c|^2}{1-|c|^2}$. It follows that, for any $f \in H^s$,
\[ \left \| S_c f \right \|_{H^s} = \left \| (1+| \xi |^2)^{s/2} \frac{(c \cdot \xi )^2}{| \xi |^2 - (c \cdot \xi )^2} \, \hat{f} ( \xi ) \right \|_{L_\xi^2} \leqslant \frac{|c|^2}{1-|c|^2} \left \| (1+| \xi |^2)^{s/2} \hat{f} ( \xi ) \right \|_{L_\xi^2} = \frac{|c|^2}{1-|c|^2} \left \| f \right \|_{H^s}. \]
This inequality gives the bound \eqref{Sc1}. For the bound \eqref{Sc2}, we first compute that 
\[ S_c (f) - S_{\tilde{c}} (f) = \mathcal{F}^{-1} \left [ \kappa_1 (c \, , \tilde{c} \, , \xi) \hat{f} ( \xi ) \right ] \] 
where
\[ \kappa_1 (c \, , \tilde{c} \, , \xi) := \frac{\left ( (c \cdot \xi )^2 - ( \tilde{c} \cdot \xi )^2 \right ) |\xi|^2}{\left ( | \xi |^2 - (c \cdot \xi )^2 \right ) \left ( | \xi |^2 - ( \tilde{c} \cdot \xi )^2 \right )} = \frac{\left ( (c - \tilde{c} ) \cdot \xi \right ) \left ( (c+ \tilde{c}) \cdot \xi \right ) | \xi |^2}{\left ( | \xi |^2 - (c \cdot \xi )^2 \right ) \left ( | \xi |^2 - ( \tilde{c} \cdot \xi )^2 \right )} . \]
We have
\[ \left | \kappa_1 (c \, , \tilde{c} \, , \xi ) \right | \leqslant \frac{2|c- \tilde{c}| \, | \xi |^4}{(1-|c|^2)|\xi |^2 (1- | \tilde{c} |^2 ) | \xi |^2} = \frac{2 |c- \tilde{c} |}{(1-|c|^2)(1-| \tilde{c} |^2)} \]
and thus
\[ \left \| S_c (f) - S_{\tilde{c}} (f) \right \|_{H^s} \leqslant \frac{2 |c- \tilde{c} |}{(1-|c|^2)(1-| \tilde{c} |^2)} \left \| f \right \|_{H^s}. \]
Combining this estimate and the bound \eqref{Sc1}, and writing $S_c (f) - S_{\tilde{c}} ( \tilde{f} ) = S_c (f) - S_{\tilde{c}} (f) + S_{\tilde{c}} (f - \tilde{f})$ thanks to the linearity of $S_c$, we obtain the bound \eqref{Sc2}. Lastly, for the bound \eqref{Sc3}, we first compute that 
\begin{equation} 
\partial_{(c_j)} S_c (f) = \mathcal{F}^{-1} \left [ \kappa_{2,j} ( c \, , \xi ) \hat{f} ( \xi ) \right ] \quad \text{where} \quad \kappa_{2,j} (c \, , \xi ) := \frac{2 | \xi |^2 ( c \cdot \xi ) \xi_j}{\left ( | \xi |^2 - (c \cdot \xi )^2 \right )^2}. 
\label{e3}
\end{equation}
We have $| \kappa_{2,j} (c \, , \xi ) | \leqslant \frac{2|c|}{(1-|c|^2)^2}$ and thus
\begin{equation} 
\left \| \partial_{(c_j)} S_c (f) \right \|_{H^s} \leqslant \frac{2|c|}{(1-|c|^2)^2} \left \| f \right \|_{H^s}. 
\label{e1}
\end{equation}
Now, we compute that $\partial_{(c_j)} S_c (f) - \partial_{(c_j)} S_{\tilde{c}} (f) = \mathcal{F}^{-1} \left [ \kappa_3 ( c \, , \tilde{c} \, , \xi ) \hat{f} ( \xi ) \right ]$ where
\[ \kappa_3 ( c \, , \tilde{c} \, , \xi ) = \frac{2 | \xi |^2 \left [ ((c - \tilde{c} ) \cdot \xi ) | \xi |^4 + (c \cdot \xi ) ( \tilde{c} \cdot \xi ) \left ( ( \tilde{c} \cdot \xi ) - (c \cdot \xi )^3 \right ) + 2 | \xi |^2 ( c \cdot \xi ) ( \tilde{c} \cdot \xi ) ((c - \tilde{c} ) \cdot \xi ) \right ] \xi_j}{\left ( | \xi |^2 - (c \cdot \xi )^2 \right )^2 \left ( |\xi |^2 - ( \tilde{c} \cdot \xi )^2 \right )^2}. \]
Writing $\left | ( \tilde{c} \cdot \xi )^3 - (c \cdot \xi )^3 \right | = \left | ((c- \tilde{c} ) \cdot \xi ) \left ( (c \cdot \xi )^2 + ( \tilde{c} \cdot \xi )^2 + (c \cdot \xi ) ( \tilde{c} \cdot \xi ) \right ) \right | \leqslant 3 |c- \tilde{c} | \, | \xi |^3$, it follows that
\[ | \kappa_3 (c \, , \tilde{c} \, , f) | \leqslant \frac{6 |c- \tilde{c}| \, | \xi |^8}{(1-|c|^2)^2 (1- | \tilde{c} |^2)^2 | \xi |^8} = \frac{6 |c- \tilde{c}|}{(1-|c|^2)^2 (1- | \tilde{c} |^2)^2}. \]
Hence,
\begin{equation}
\left \| \partial_{(c_j)} S_c (f) - \partial_{(c_j)} S_{\tilde{c}} (f) \right \|_{H^s} \leqslant \frac{6 |c- \tilde{c}|}{(1-|c|^2)^2 (1- | \tilde{c} |^2)^2} \left \| f \right \|_{H^s}. 
\label{e2}
\end{equation}
Combining the estimates \eqref{e1} and \eqref{e2}, and writing $\partial_{(c_j)} S_c (f) - \partial_{(c_j)} S_{\tilde{c}} ( \tilde{f} ) = \partial_{(c_j)} S_c (f) - \partial_{(c_j)} S_{\tilde{c}} (f) + \partial_{(c_j)} S_{\tilde{c}} (f - \tilde{f})$, we obtain the last bound \eqref{Sc3}.
\end{proof}

Now, thanks to a simple transformation, we may assume that $c = c_1 e_1$. Indeed, for any $\theta \in \R$, define the rotation $R_\theta = \left ( \begin{array}{cc} \cos \theta & \sin \theta \\ - \sin \theta & \cos \theta \end{array} \right )$, the functions $\tilde{U} ( \tilde{y} ) = U(R_\theta \tilde{y})$, $\tilde{N} ( \tilde{y} ) = N(R_\theta \tilde{y})$, $\tilde{V} ( \tilde{y} ) = R_{- \theta} V (R_\theta \tilde{y} )$ and finally the new celerity $\tilde{c} = R_{- \theta} c$. Straightforward computations show that $(U \, , N \, , V)$ satisfies \eqref{sysNUV} if and only $(\tilde{U} \, , \tilde{N} \, , \tilde{V})$ satisfies
\[ \left \{ \begin{array}{l} \Delta \tilde{U} = \tilde{U} + \tilde{N} \tilde{U} \\ \Delta \left ( \tilde{N} + | \tilde{U} |^2 \right ) = \tilde{c} \cdot \nabla \tilde{V} \\ \tilde{c} \cdot \nabla \tilde{N} = \nabla \cdot \tilde{V} \end{array} \right. \]
which is the same system as \eqref{sysNUV}, where the celerity $c$ has been replaced by $\tilde{c}$. Henceforth, thanks to this change of variable, one can assume that $c = c_1 e_1$. This assumption will be convenient for symmetry purposes. We simply denote $c=c_1$. 

We shall write that a function $f(y_1 \, , y_2)$ is \textit{$y_1$-even (resp. $y_1$-odd)} when, for all $y_2 \in \R$, the function $y_1 \mapsto f(y_1 \, , y_2)$ is even (resp. odd). The same convenient vocabulary will be used for $y_2$ instead of $y_1$. Now that $c=ce_1$, note that the operator $S_c = c^2 \left ( (1-c^2) \partial_{y_1}^2 + \partial_{y_2}^2 \right )^{-1} \circ \partial_{y_1}^2$ preserves symmetry properties as follows. Take $j \in \{ 1 \, , 2 \}$.
\begin{itemize}
	\item[$\bullet$] If $f$ is $y_j$-even, then $S_c f$ is $y_j$-even as well. 
	\item[$\bullet$] If $f$ is $y_j$-odd, then $S_c f$ is $y_j$-odd as well.
\end{itemize}

From system \eqref{sysNUV2} we see that
\begin{equation}
- \Delta U + U - |U|^2 U = S_c \left ( |U|^2 \right ) U.
\label{eqU}
\end{equation}
We look for a solution $U$ of \eqref{eqU} that would be close to the NLS soliton $Q$. Hence we write $U=Q+ \eta_1 + i \eta_2$ where $\eta_1$ and $\eta_2$ are real-valued. Using the identity $\Delta Q = Q-Q^3$ we find that $U$ solves \eqref{eqU} if and only if $(\eta_1 \, , \eta_2)$ solves the following system:
\begin{equation}
\begin{split}
L_+ \eta_1 = F_c^+ ( \eta_1 \, , \eta_2 ) := & \, S_c(Q^2)Q + S_c (Q^2) \eta_1 + 2 S_c (Q \eta_1) Q  + 3Q \eta_1^2 + Q \eta_2^2 \\
& \qquad + 2 S_c (Q \eta_1 ) \eta_1 + S_c ( \eta_1^2 ) Q + S_c ( \eta_2^2) Q + \eta_1^3 + \eta_1 \eta_2^2 \\
& \qquad + S_c ( \eta_1^2) \eta_1 + S_c ( \eta_2^2 ) \eta_1 \\
\text{and} \quad L_- \eta_2 = F_c^- ( \eta_1 \, , \eta_2 ) := & \, S_c (Q^2) \eta_2 + 2 Q \eta_1 \eta_2 + 2S_c ( Q \eta_1) \eta_2 + \eta_1^2 \eta_2 + \eta_2^3 \\
& \qquad + S_c ( \eta_1^2 ) \eta_2 + S_c ( \eta_2^2 ) \eta_2.
\end{split}
\label{syseta}
\end{equation}

Define 
\begin{align*}
& H_r^2 = \left \{ \eta_1 \in H^2 ( \R^2 ) \, \, \text{such that $\eta_1$ is radial} \right \} , \\
& H_{oe}^2 = \left \{ \eta_2 \in H^2 ( \R^2 ) \, \, \text{such that $\eta_2$ is $y_1$-odd and $y_2$-even} \right \} \\
\text{and} \quad & E = H_r^2 \times H_{oe}^2 .
\end{align*} 
The fact that $H^2 ( \R^2 )$ is an algebra and Lemma \ref{lemSc} ensure that, for $( \eta_1 \, , \eta_2 ) \in H^2 ( \R^2 )^2$, $F_c^{\pm} ( \eta_1 \, , \eta_2 ) \in H^2 ( \R )$. Furthermore, an analysis of the symmetries in the expressions of $F_c^\pm$ shows that
\begin{equation}
(\eta_1 \, , \eta_2) \in E \quad \Longrightarrow \quad F_c^{\pm} ( \eta_1 \, , \eta_2 ) \in E. 
\label{stabF}
\end{equation}
We now recall the following estimates.

\begin{lemma} \label{lemL}
The following properties hold.
\begin{itemize}
	\item[$\bullet$] For any $f_1 \in H_r^2$, there exists a unique $g_1 = L_+^{-1} f_1 \in H_r^2$ such that $L_+g_1=f_1$. Besides,
	\[ ||L_+^{-1} f_1 ||_{H^2} \lesssim ||f_1||_{L^2}. \]
	\item[$\bullet$] For any $f_2 \in H_{oe}^2$, there exists a unique $g_2 = L_-^{-1} f_2 \in H_{oe}^2$ such that $L_-g_2=f_2$. Besides,
	\[ ||L_-^{-1} f_2 ||_{H^2} \lesssim ||f_2||_{L^2}. \]
\end{itemize}
\end{lemma}

\begin{proof}
For $f_1 \in H_r^2$ and $f_2 \in H_{oe}^2$, the following orthogonality relations hold (and are necessary to invert the operators $L_\pm$, see \cite{We2}):
\[ \int_{\R^2} ( \partial_{y_1} Q )  f_1 = \int_{\R^2} ( \partial_{y_2} Q ) f_1 = 0 \quad \text{and} \quad \int_{\R^2} Q f_2 = 0. \]
The first point of the lemma is proven in \cite{We2} and \cite[Propositions 4.1 and 4.5, Lemma 4.6]{GM}. For the second point, the uniqueness comes from the fact that $H_{oe}^2 \cap \text{Ker} (L_-) = H_{oe}^2 \cap \text{span} (Q) = \{ 0 \}$, see \cite[Proposition 2.8]{We2} for the description of the kernel of $L_-$. The proof of the existence works similarly to the one for $L_+$, although the function $f_2$ is not radial; see the Appendix for a detailed adaptation of the proof to the current non-radial situation.
\end{proof}

The previous lemma and the property \eqref{stabF} allow to define, for all $(\eta_1 \, , \eta_2 ) \in E$,
\[ \mathbf{G}_c ( \eta_1 \, , \eta_2 ) = \left ( G_c^+ ( \eta_1 \, , \eta_2 ) \, , G_c^- ( \eta_1 \, , \eta_2 ) \right ) := \left ( L_+^{-1} F_c^+ ( \eta_1 \, , \eta_2 ) \, , L_-^{-1} F_c^- ( \eta_1 \, , \eta_2 ) \right ) \]
so that system \eqref{syseta} is equivalent to the following fixed point problem in the space $E$:
\begin{equation} 
( \eta_1 \, , \eta_2 ) = \mathbf{G}_c ( \eta_1 \, , \eta_2 ). 
\label{sysG}
\end{equation}

We now state our fixed point result. For simplicity, we write $E_\delta := E \cap \overline{B}_{H^2 \times H^2} ( 0 \, , \delta )$ and $\left \| (g_1 \, , g_2) \right \|_E := \left \| g_1 \right \|_{H^2} + \left \| g_2 \right \|_{H^2}$. Equipped with this norm, the space $E_\delta$ is a complete metric space.

\begin{proposition} \label{propFP}
There exists $c_* \in (0 \, , 1)$ and $\delta_0 > 0$ such that, for all $c \in (-c_* \, , c_*)$, there exists a unique couple $( \eta_1 \, , \eta_2 ) \in E_{\delta_0}$ solution of \eqref{sysG}. Moreover, 
\begin{itemize}
	\item the map $c \in (-c_* \, , c_*) \mapsto ( \eta_1 \, , \eta_2 ) \in E$ is $\mathscr{C}^1$;
	\item $\left \| ( \eta_1 \, , \eta_2 ) \right \|_E \lesssim c^2$ ;
	\item $\eta_1$ and $\eta_2$ belong to $H^s$ for all $s \geqslant 0$.
\end{itemize}
\end{proposition}

\begin{proof}
We proceed to show that, for $\delta_0 >0$ and $|c|$ small enough, the application $\mathbf{G}_c$ is a contraction from $E_{\delta_0}$ to $E_{\delta_0}$. In all that follows, take $c$ such that $|c| \leqslant c_* \leqslant \frac{1}{2}$ (where $c_*$ shall be defined later). In particular, $|||S_c||| \leqslant \frac{c_*^2}{2}$ (see \eqref{Sc1}).
\begin{itemize}
	\item[$\bullet$] Stability of the ball. Take $(\eta_1 \, , \eta_2) \in E_{\delta_0}$. Recall the Sobolev embedding $\left \| \eta_j \right \|_{L^\infty} \lesssim \left \| \eta_j \right \|_{H^2} \lesssim \delta_0$. Using this Sobolev embedding, Cauchy-Schwarz inequalities and the bound \eqref{Sc1}, it follows from the definitions of $F_c^{\pm}$ in \eqref{syseta} that
	\begin{align*}
	& \left \| F_c^+ ( \eta_1 \, , \eta_2 ) \right \|_{L^2} \lesssim c^2 + \delta_0^2 + \delta_0^3 \lesssim c_*^2 + \delta_0^2 \\
	\text{and} \quad & \left \| F_c^- ( \eta_1 \, , \eta_2 ) \right \|_{L^2} \lesssim c^2 \delta_0 + \delta_0^2 + \delta_0^3 \lesssim c_*^2 \delta_0 + \delta_0^2.
	\end{align*}
	Taking $c_*>0$ small enough (depending on $\delta_0$, namely $c_* \lesssim \delta_0$), we obtain $\left \| F_c^{\pm} ( \eta_1 \, , \eta_2 ) \right \|_{L^2} \lesssim \delta_0^2$. Applying Lemma \ref{lemL}, it follows that, for $\delta_0 > 0$ small enough,
	\[ \left \| \mathbf{G}_c ( \eta_1 \, , \eta_2 ) \right \|_E \lesssim \delta_0^2 \quad \text{thus} \quad \left \| \mathbf{G}_c ( \eta_1 \, , \eta_2 ) \right \|_E < \delta_0 \quad \text{i.e.} \quad \mathbf{G}_c ( \eta_1 \, , \eta_2 ) \in E_{\delta_0}. \]
	\item[$\bullet$] Contraction property. Take $\eta = (\eta_1 \, , \eta_2)$ and $\tilde{\eta} = ( \tilde{\eta}_1 \, , \tilde{\eta}_2 )$ in $E_{\delta_0}$. As above, we use the Sobolev embedding $H^2 \to L^\infty$, Cauchy-Schwarz inequalities and the bound \eqref{Sc1} to control the different terms in $F_c^{\pm} ( \eta ) - F_c^{\pm} ( \tilde{\eta} )$, see definitions in \eqref{syseta}. For instance,
	\begin{align*}
	\left \| S_c ( \eta_2^2 ) \eta_1 - S_c ( \tilde{\eta}_2^2 ) \tilde{\eta}_1 \right \|_{L^2} \leqslant & \left \| S_c ( \eta_2^2 ) \eta_1 - S_c ( \eta_2^2 ) \tilde{\eta}_1 \right \|_{L^2} + \left \| S_c ( \eta_2^2 ) \tilde{\eta}_1 - S_c ( \tilde{\eta}_2^2 ) \tilde{\eta}_1 \right \|_{L^2} \\
	\leqslant & \left \| S_c ( \eta_2^2 ) \right \|_{L^2} \left \| \eta_1 - \tilde{\eta}_1 \right \|_{L^2} + \left \| S_c ( \eta_2^2 - \tilde{\eta}_2^2 ) \right \|_{L^2} \left \| \tilde{\eta}_1 \right \|_{L^2} \\
	\lesssim & \left \| \eta_2^2 \right \|_{L^2} \left \| \eta - \tilde{\eta} \right \|_E + \left \| \eta_2^2 - \tilde{\eta}_2^2 \right \|_{L^2} \delta_0 \\
	\lesssim & \left \| \eta_2 \right \|_{L^\infty} \left \| \eta_2 \right \|_{L^2} \left \| \eta - \tilde{\eta} \right \|_E + \left \| \eta_2 - \tilde{\eta}_2 \right \|_{L^2} \left \| \eta_2 + \eta_2 \right \|_{L^2} \delta_0 \\
	\lesssim & \, \delta_0^2 \left \| \eta - \tilde{\eta} \right \|_E.
	\end{align*}
	The other terms are controlled similarly and more easily. We find that
	\[ \left \| F_c^{\pm} ( \eta - \tilde{\eta} ) \right \|_{L^2} \lesssim (c^2 + \delta_0 + \delta_0^2) \left \| \eta - \tilde{\eta} \right \| \lesssim \delta_0 \left \| \eta - \tilde{\eta} \right \|_E  \]
	since we recall that $c_* \lesssim \delta_0$. Taking $\delta_0 > 0$ small enough, it follows from Lemma \ref{lemL} that
	\[ \left \| \mathbf{G}_c ( \eta ) - \mathbf{G}_c ( \tilde{\eta} ) \right \|_E \lesssim \delta_0 \left \| \eta - \tilde{\eta} \right \|_E \quad \text{hence} \quad \left \| \mathbf{G}_c ( \eta ) - \mathbf{G}_c ( \tilde{\eta} ) \right \|_E \leqslant \frac{1}{2} \left \| \eta - \tilde{\eta} \right \|_E. \]
\end{itemize}
Provided that $\delta_0 > 0$ is taken small enough, and then $c_* > 0$ is taken small enough accordingly, we have shown that the map $\mathbf{G}_c$ defines a contraction $E_{\delta_0} \to E_{\delta_0}$. Applying Banach's fixed point theorem, there exists a unique couple $(\eta_1 \, , \eta_2 ) \in E_{\delta_0}$ solution of \eqref{syseta}. Now we prove the three remaining points of Proposition \ref{propFP}.
\begin{itemize}
	\item[$\bullet$] First let us prove that the map $c \mapsto ( \eta_1^{(c)} \, , \eta_2^{(c)} )$ is $\mathscr{C}^1$, where $(\eta_1^{(c)} \, , \eta_2^{(c)})$ denotes the solution of \eqref{syseta} constructed above. We already know that the contraction coefficient of $\mathbf{G}_c$ does not depend on $c$. To establish the desired result, it suffices to show that the application
	\[ \left ( \begin{array}{ccl} (-c_* \, , c_*) \times E_{\delta_0} & \to & E_{\delta_0} \\ (c \, , \eta_1 \, , \eta_2 ) & \mapsto & \mathbf{G}_c ( \eta_1 \, , \eta_2 ) \end{array} \right ) \]
	is $\mathscr{C}^1$. For example,
	\begin{align*} 
	\partial_{(\eta_1)} G_c^+ ( \eta_1 \, , \eta_2 ) [h] = & \, L_-^{-1} \left [ S_c (Q^2) h + 2S_c ( Q h ) Q + 6Q \eta_1 h + 2 S_c ( Q h ) \eta_1 + 2S_c (Q \eta_1)h  \right. \\
	& \qquad \quad \left. + 2 S_c ( \eta_1 h ) Q + 3 \eta_1^2 h + \eta_2^2 h + 2 S_c ( \eta_1 h ) \eta_2 + S_c ( \eta_1^2 ) h + S_c ( \eta_2^2 ) h \right ]. 
	\end{align*}
	We proceed similarly as we did for the contraction property and find that
	\[ \left \| \partial_{(\eta_1)} G_c^+ ( \eta_1 \, , \eta_2 ) [h] - \partial_{(\eta_1)} G_c^+ ( \tilde{\eta}_1 \, , \tilde{\eta}_2 ) [h] \right \|_E \lesssim \left \| \eta - \tilde{\eta} \right \|_E \left \| h \right \|_E. \]
	More generally, analogous computations show that, for $j \in \{ 1 \, , 2 \}$,
	\[ \left \| \partial_{(\eta_j)} G_c^{\pm} ( \eta_1 \, , \eta_2 ) [h] - \partial_{(\eta_j)} G_c^{\pm} ( \tilde{\eta}_1 \, , \tilde{\eta}_2 ) [h] \right \|_E \lesssim \left \| \eta - \tilde{\eta} \right \|_E \left \| h \right \|_E \]
	which proves that the applications $\eta \in E_{\delta_0} \mapsto \partial_{(\eta_j)} G_c^{\pm} ( \eta )$ are continuous. Using \eqref{Sc2}, it follows that the applications $(c \, , \eta ) \in (-c_* \, , c_*) \times E_{\delta_0} \mapsto \partial_{(\eta_j)} G_c^{\pm} ( \eta )$ are continuous. Lastly, the continuous differentiability with regards to $c$ follows easily from the bound \eqref{Sc3} and the explicit expression \eqref{e3}. Henceforth, the map $(c \, , \eta ) \mapsto \mathbf{G}_c ( \eta )$ defines a $\mathscr{C}^1$ function from $(-c_* \, , c_*) \times E_{\delta_0}$ to $E_{\delta_0}$. This suffices to show that the map $c \mapsto (\eta_1^{(c)} \, , \eta_2^{(c)})$ is $\mathscr{C}^1$ from $(-c_* \, ,c_*)$ to $E$.
	\item[$\bullet$] Now let us show that $\left \| ( \eta_1^{(c)} \, , \eta_2^{(c)} ) \right \|_E \lesssim c^2$. We simply write $\eta_1 = \eta_1^{(c)}$ and $\eta_2 = \eta_2^{(c)}$. We have
	\begin{align*} 
	\left \| ( \eta_1 \, , \eta_2 ) \right \|_E - \left \| \mathbf{G}_c (0 \, , 0) \right \|_E \leqslant & \left \| ( \eta_1 \, , \eta_2 ) - \mathbf{G}_c ( 0 \, , 0 ) \right \|_E = \left \| \mathbf{G}_c ( \eta_1 \, , \eta_2 ) - \mathbf{G}_c (0 \, , 0) \right \|_E \\
	\leqslant & \, \frac{1}{2} \left \| ( \eta_1 \, , \eta_2 ) - (0 \, , 0) \right \|_E = \frac{1}{2} \left \| ( \eta_1 \, , \eta_2 ) \right \|_E 
	\end{align*}
	thus
	\[ \left \| ( \eta_1 \, , \eta_2 ) \right \|_E \leqslant 2 \left \| \mathbf{G}_c (0 \, , 0) \right \|_E = 2 \left \| L_+^{-1} \left ( S_c (Q^2) Q \right ) \right \|_{H^2} \lesssim c^2. \]
	\item[$\bullet$] Now, let us prove that $\eta_1$ and $\eta_2$ belong to $H^s$ for any $s \geqslant 0$. We already know that $\eta_1 , \eta_2 \in H^2$. Take $s \geqslant 2$ and assume that $\eta_1 , \eta_2 \in H^s$. We know that $H^s ( \R^2 )$ is an algebra and is stable by $S_c$. As a consequence, $F_c^+ ( \eta_1 \, , \eta_2 ) \in H^s$ and $F_c^- ( \eta_1 \, , \eta_2 ) \in H^s$ as well (see \eqref{syseta} for the definitions of $F_c^{\pm}$). It follows that
	\[ \Delta \eta_1 = (1-3Q^2) \eta_1 - F_c^+ ( \eta_1 \, , \eta_2) \in H^s \quad \text{and} \quad \Delta \eta_2 = (1-Q^2) \eta_2 - F_c^- ( \eta_1 \, , \eta_2 ) \in H^s. \]
	Recalling that $\left \| \eta_j \right \|_{H^{s+2}} \lesssim \left \| \eta_j \right \|_{H^s} + \left \| \Delta \eta_j \right \|_{H^s}$, it follows that $\eta_1$ and $\eta_2$ belong to $H^{s+2}$. Iterating this argument, we find that $\eta_1$ and $\eta_2$ belong to $H^s$ for all $s \geqslant 0$.
\end{itemize}
This concludes the proof.
\end{proof}

\begin{remark} \label{rkcel}
The fact that the celerity $c_*$ depends on $\delta_0$ is due only to the presence of the "constant term" $S_c (Q^2)Q$ in $F_c^+$: to have $\left \| L_+^{-1} \left ( S_c (Q^2)Q \right ) \right \|_{H^2} \lesssim \delta_0$ we require $c_* \lesssim \delta_0$. 
\end{remark}

We can now conclude the proof of Theorem \ref{thm1}, except for the estimates \eqref{decays}. Take $c \in (-c_* \, , c_*)$. Since $Q$, $\eta_1^{(c)}$ and $\eta_2^{(c)}$ belong to $H^s$ for all $s \geqslant 0$, so does the function $U_c = Q + \eta_1^{(c)} + i \eta_2^{(c)}$. The functions $N_c = -|U_c|^2 - S_c (|U_c|^2)$ and $V_c = \left ( \begin{array}{c} V_{c,1} \\ V_{c,2} \end{array} \right ) = \left ( \begin{array}{c} T_{c,1} (|U_c|^2) \\ T_{c,2} (|U_c|^2) \end{array} \right )$ also belong to $H^s$ for all $s \geqslant 0$, thanks to Lemma \ref{lemSc}. Moreover, using Proposition \ref{propFP}, we see that
\[ ||U_c - Q||_{H^2} \lesssim || \eta_1^{(c)} ||_{H^2} + || \eta_2^{(c)} ||_{H^2} \lesssim c^2. \]
In particular, $||U_c||_{H^2} \lesssim 1$. Similarly, using also \eqref{Sc1},
\begin{align*}
|| N_c + Q^2 ||_{H^2} \lesssim & \, || \, |U_c|^2 - Q^2 ||_{H^2} + || S_c (|U_c|^2) ||_{H^2} \\
\leqslant & \, ||U_c \overline{U}_c - Q \overline{U}_c ||_{H^2} + ||Q \overline{U}_c - Q^2 ||_{H^2} + ||S_c (|U_c|^2) ||_{H^2} \\
\lesssim & \, || \overline{U}_c ||_{H^2} || U_c-Q||_{H^2} + ||Q||_{H^2} || \overline{U}_c - Q ||_{H^2} + |c|^2 ||U_c||_{H^2}^2 \\
\lesssim & \, 1 \times c^2 + 1 \times c^2 + c^2 \times 1 \lesssim c^2.
\end{align*}
Similarly, using the estimates \eqref{Sc1} about $T_{c,1}$ and $T_{c,2}$, we prove that $||V_{c,1}||_{H^2} \lesssim c$ and $|| V_{c,2}||_{H^2} \lesssim c$. This concludes the proof of Theorem \ref{thm1}. \hfill \qedsymbol

\begin{remark} \label{rknorms}
We deduce from Theorem \ref{thm1} that, in particular, $||U_c||_{H^2} + ||N_c||_{H^2} + ||V_c||_{H^2 \times H^2} \lesssim 1$ thus $||U_c||_{L^\infty} + ||N_c||_{L^\infty} + ||V_c||_{L^\infty \times L^\infty} \lesssim 1$ by Sobolev embedding.
\end{remark}

\section{Asymptotic behavior of the solitons} \label{sec2}

From now on, take $|c| < c_*$. We investigate the decrease and asymptotic behavior of the solitary waves constructed in the previous section. First, like the standard NLS solitons, $U_c$ and its derivatives decrease exponentially at infinity. 

\begin{lemma} \label{lemAgmon}
There exists $\tilde{c}_* >0$ such that, for all $|c| < \tilde{c}_*$, $m \in \N^2$ and $y \in \R^2$,
\begin{equation}
| \partial^m U_c (y) | \lesssim_{m} e^{- \frac{1}{2} |y|}.
\label{agmon}
\end{equation}
\end{lemma}

\begin{remark} \label{rkdemi}
Clearly, the proof below can be easily adapted to show that $| \partial^m U_c (y) | \lesssim_m e^{- 1^- |y|}$ where $1^-$ is any $\alpha \in (0 \, , 1)$.
\end{remark}

\begin{proof}
This proof relies on so-called Agmon arguments, in reference to \cite{Ag}. See \cite[Lemma 2.4]{KMR} for a variant of the proof of Agmon, that we adapt here. From \eqref{sysNUV2} we know that $(- \Delta +1)U_c = -N_cU_c$, where $N_c$ and $U_c$ are functions that belong to $H^s$ for all $s \geqslant 0$. It is well-known (see \cite[Chap. 4.3, Ex. 1]{Ev} for instance) that, in $\R^2$, the solution $f$ to $(- \Delta + 1)f=g$ is given by $f(x) = \int_{\R^2} \mathcal{K} (x-y) g(y) \, \text{d}y$ where $\mathcal{K}$ is a modified Bessel function of the second kind, which satisfies
\[ | \mathcal{K} (y) | \lesssim (1 + \left | \ln |y| \, \right | ) e^{- |y|} . \]
Hence
\[ U_c(x) = -\int_{\R^2} \mathcal{K} (x-y) N_c(y) U_c(y) \, \text{d}y. \]
Take $L>0$ and $M>0$, to be specified later. We shall be careful that $L$ does not depend on $M$. 
\begin{itemize}
	\item[$\bullet$] First, note that $e^{\frac{1}{2} |x|} \leqslant e^{\frac{1}{2} |x-y|} e^{\frac{1}{2} |y|}$. It follows that
	\begin{align*}
	& \min \left ( M \, , e^{\frac{1}{2} |x|} \right ) \left | \int_{|y|<L} \mathcal{K} (x-y) N_c(y) U_c(y) \, \text{d}y \right | \\
	& \lesssim e^{\frac{1}{2} |x|} \int_{|y|<L} (1+|\ln|x-y| \, | ) e^{-  |x-y|} \| N_c \|_{L^\infty} \| U_c \|_{L^\infty} \, \text{d}y \\
	& \lesssim \int_{|y|<L} e^{\frac{1}{2} |y|} e^{- \frac{1}{2} |x-y|} (1+|\ln |x-y| \, |) \, \text{d}y \\
	& \lesssim \int_{|y|<L} e^{\frac{L}{2}} e^{- \frac{1}{2} |x-y|} (1+ | \ln |x-y| \, |) \, \text{d}y \\
	& \lesssim e^{\frac{L}{2}} \int_{\R^2} e^{- \frac{1}{2} |z|} (1+ | \ln |z| \, | ) \, \text{d}z \\
	& \leqslant C_L < \infty .
	\end{align*}
	\item[$\bullet$] Second, note that $M \leqslant M e^{\frac{1}{2} |x-y|}$ thus $\min (M \, , e^{\frac{1}{2} |x|}) \leqslant e^{\frac{1}{2} |x-y|} \min \left ( M \, , e^{\frac{1}{2} |y|} \right )$. It follows that
	\begin{align*}
	& \min \left ( M \, , e^{\frac{1}{2} |x|} \right ) \left | \int_{|y| \geqslant L} \mathcal{K} (x-y) N_c(y) U_c(y) \, \text{d}y \right | \\
	& \lesssim \int_{|y| \geqslant L} e^{\frac{1}{2} |x-y|} \min \left ( M \, , e^{\frac{1}{2} |y|} \right ) (1+ | \ln |x-y| \, |) e^{- |x-y|} \left ( \sup\limits_{|z| \geqslant L} |N_c(z)| \right ) |U_c(y)| \, \text{d}y \\
	& \lesssim \left ( \sup\limits_{|z| \geqslant L} |N_c(z)| \right ) \int_{|y| \geqslant L} e^{- \frac{1}{2} |x-y|} (1+ | \ln |x-y| \, |) \left ( \sup\limits_{|z| \geqslant L} \left [ \min \left ( M \, , e^{\frac{1}{2} |z|} \right ) |U_c(z)| \right ] \right ) \, \text{d}y \\
	& \lesssim \left ( \sup\limits_{|z| \geqslant L} |N_c(z)| \right ) \left ( \sup\limits_{|z| \geqslant L} \left [ \min \left ( M \, , e^{\frac{1}{2} |z|} \right ) |U_c(z)| \right ] \right ) \int_{\R^2} e^{- \frac{1}{2} |z|} (1+| \ln |z| \, | ) \, \text{d}z \\
	& \lesssim \left ( \sup\limits_{|z| \geqslant L} |N_c(z)| \right ) \left ( \sup\limits_{z \in \R^2} \left [ \min \left ( M \, , e^{\frac{1}{2} |z|} \right ) |U_c(z)| \right ] \right ).
	\end{align*}
	Recall from Theorem \ref{thm1} that $||N+Q^2||_{L^\infty} \lesssim ||N+Q^2||_{H^2} \lesssim c^2$ via Sobolev embedding. It follows that
	\[ \sup\limits_{|z| \geqslant L} |N_c(z)| \leqslant \left ( \sup\limits_{|z| \geqslant L} Q^2 \right ) + ||N+Q^2||_{L^\infty} \lesssim e^{-L} + c^2 . \]
	Thus, taking $\tilde{c}_* \in (0 \, , c_* )$ small enough, $|c| < \tilde{c}_*$ and $L>0$ large enough (not depending on $c$ nor $M$),
	\[ \min \left ( M \, , e^{\frac{1}{2} |x|} \right ) \left | \int_{|y| \geqslant L} \mathcal{K} (x-y) N_c(y) U_c(y) \, \text{d}y \right | \leqslant \frac{1}{2} \sup\limits_{z \in \R^2} \left [ \min \left ( M \, , e^{\frac{1}{2} |z|} \right ) |U_c(z)| \right ] . \]
	We fix such an $L$ and recall that it does not depend on $c$ nor $M$. 
\end{itemize}
Combining the previous estimates, we find that, for all $x \in \R^2$,
\begin{align*}
\min \left ( M \, , e^{\frac{1}{2} |x|} \right ) |U_c(x)| \leqslant & \min \left ( M \, , e^{\frac{1}{2} |x|} \right ) \left | \int_{|y| \geqslant L} \mathcal{K} (x-y) N_c(y) U_c(y) \, \text{d}y \right | \\
& \qquad + \min \left ( M \, , e^{\frac{1}{2} |x|} \right ) \left | \int_{|y| < L} \mathcal{K} (x-y) N_c(y) U_c(y) \, \text{d}y \right | \\
\leqslant & \, \frac{1}{2} \sup\limits_{z \in \R^2} \left [ \min \left ( M \, , e^{\frac{1}{2} |z|} \right ) |U_c(z)| \right ] + C.
\end{align*}
Since $\sup\limits_{x \in \R^2} \left [ \min \left ( M \, , e^{\frac{1}{2} |x|} \right ) |U_c(x)| \right ] \leqslant M \| U_c \|_{L^\infty} < \infty$, taking the supremum over all $x \in \R^2$ in the inequality above leads to $\min \left ( M \, , e^{\frac{1}{2} |x|} \right ) |U_c(x)| \leqslant 2C$. Recall that the constant $C$ does not depend on $M$. Letting $M \to + \infty$, it follows that
\[ \forall x \in \R^2 , \quad |U_c (x)| \lesssim e^{- \frac{1}{2} |x|}. \]
Now, for the derivatives of $U_c$, take $j \in \{ 1 \, , 2 \}$. From \eqref{sysNUV2} we see that the equation satisfied by $\partial_{y_j} U_c$ is $(- \Delta + 1) ( \partial_{y_j} U_c ) = -N_c \partial_{y_j} U_c - U_c \partial_{y_j} N_c$. The only difference with the previous proof is the last term $U_c \partial_{y_j} N_c$, which does not contain $\partial_{y_j} U_c$ but which is already exponentially decreasing. Explicitly, on the one hand,
\begin{align*}
& \min \left ( M \, , e^{\frac{1}{2} |x|} \right ) \left | \int_{\R^2} \mathcal{K} (x-y) U_c (y) \partial_{y_j} N_c (y) \, \text{d}y \right | \\
& \lesssim e^{\frac{1}{2} |x|} \int_{\R^2} (1+|\ln|x-y| \, |) e^{- |x-y|} \underbrace{\left \| \partial_{y_j} N_c \right \|_{L^\infty}}_{\lesssim 1} e^{- \frac{1}{2} |y|} \, \text{d}y \\
& \lesssim \int_{\R^2} e^{\frac{1}{2} |y|} (1+|\ln |x-y| \, |) e^{- \frac{1}{2} |x-y|} e^{- \frac{1}{2} |y|} \, \text{d}y \\
& \lesssim \int_{\R^2} e^{- \frac{1}{2} |z|} (1+ | \ln |z| \, | ) \, \text{d}z \\
& \lesssim 1.
\end{align*}
And on the other hand, as before (splitting between $|y|<L$ and $|y| \geqslant L$ for a suitable $L$ that does not depend on $M$),
\[ \min \left ( M \, , e^{\frac{1}{2} |x|} \right ) \left | \int_{\R^2} \mathcal{K} (x-y) N_c(y) \partial_{y_j} U_c(y) \, \text{d}y \right | \leqslant \frac{1}{2} \sup\limits_{z \in \R^2} \left [ \min \left ( M \, , e^{\frac{1}{2} |z|} \right ) |\partial_{y_j} U_c(z)| \right ] + C \]
where the constant $C$ does not depend on $M$. It follows that, for all $x \in \R^2$,
\[ \min \left ( M \, , e^{\frac{1}{2} |x|} \right ) |\partial_{y_j} U_c (x)| \leqslant \frac{1}{2} \sup\limits_{x \in \R^2} \left [ \min \left ( M \, , e^{\frac{1}{2} |z|} \right ) |\partial_{y_j} U_c (z)| \right ] + C' \]
where the constant $C'$ does not depend on $M$. The end of the proof is the same as previously, and we obtain
\[ \forall x \in \R^2 , \quad | \partial_{y_j} U_c | \lesssim e^{- \frac{1}{2} |x|}. \]
We proceed similarly for higher derivatives: adapting the lines above to the equation $(- \Delta + 1) \partial_{y_j}^k U_c = - \partial_{y_j}^k (N_cU_c) = - \sum\limits_{\ell=0}^k \binom{k}{\ell} ( \partial_{y_j}^{\ell} N_c ) ( \partial_{y_j}^{k- \ell } U_c )$, we conclude by inductive reasoning (on $k$) that $| \partial_{y_j}^k U_c | \lesssim e^{- \frac{1}{2} |x|}$. Due to Leibnitz's product formula, the implicit constant contained in the inequality $\lesssim$ depends on $k$. The estimate for $\partial_{y_1}^{m_1} \partial_{y_2}^{m_2} U_c$ follows similar arguments.
\end{proof}

Contrary to $U_c$, the functions $N_c$ and $V_c$ do not exponentially decrease. The asymptotics stated below differ from the Schrödinger case and from the one-dimensional Zakharov solitary waves. The non-exponential decrease of $N_c$ and $V_c$ makes the construction of two-dimensional multi-solitons more challenging than in one dimension (see \cite{GR3}). In order to construct such multi-solitons, the more precise asymptotics below should be needed (see \cite{KMR} for a similar approach).

\begin{lemma} \label{lemDecrNV}
For all $|c| < \tilde{c}_*$, $m \in \N^2$ and $|y| \geqslant 4$,
\begin{align}
& | \partial^m N_c (y) | \lesssim_m   e^{- |y|} + \frac{c^2}{|y|^{|m|+2}} \label{decrN} \\
\text{and} \quad & | \partial^m V_{c,1} (y) | + | \partial^m V_{c,2} (y) | \lesssim_{m} \frac{c}{|y|^{|m|+2}} . \label{decrV}
\end{align}
Besides, if $c=ce_1$, the following expansion holds for all $K \geqslant 3$ and $|y| \geqslant 4$,
\begin{equation}
\begin{split}
& \left | N_c (y) + \frac{c^2}{4 \pi^2 \nu^2} \sum\limits_{n=0}^K \frac{1}{|z|^{2n+2}} \int_{\R^2} \left ( 2 - \frac{4(n+1) (z_1- \zeta_1)^2}{|z|^2} \right ) (2z \cdot \zeta - | \zeta |^2 )^n |U_c|^2 ( \nu \zeta_1 \, , \zeta_2 ) \, \text{d} \zeta \right | \\
& \lesssim \frac{c^2 K3^K}{|y|^{K+1}} + e^{- |y|} + c^2 C_{c,K} e^{- \frac{\nu^2}{8} |y|^{1/K}} 
\end{split}
\label{expanN}
\end{equation}
where $\nu = \sqrt{1-c^2}$, $z = (z_1 \, , z_2) = (\nu^{-1} y_1 \, , y_2)$ and $C_{c,K} = (2K+5)! (1+8 \nu^{-2})^{2K+3}$.
\end{lemma}

\begin{remark} \label{rkdemi2}
As discussed previously in Remark \ref{rkdemi}, the proof below can be adapted to show that $| \partial^m N_c | \lesssim_m e^{- 2^- |y|} + \frac{c^2}{|y|^{|m|+2}}$. 
\end{remark}

\begin{proof}
We recall that $N_c = -|U_c|^2-S_c(|U_c|^2)$ and we take $c = ce_1$ (without loss of generality). The main part of the following proof is devoted to the analysis of $S_c(h)$ where $h = |U_c|^2$ satisfies $| \partial_{y_j}^k h(y)| \lesssim e^{- \frac{1}{2} |y|}$. Recall that $S_c = \Delta_c^{-1} (c \cdot \nabla )^2 = c^2 ((1-c^2) \partial_{y_1}^2 + \partial_{y_2}^2)^{-1} \partial_{y_1}^2$. The elementary solution to $\Delta_c = (1-c^2) \partial_{y_1}^2 + \partial_{y_2}^2$ in two dimensions is $- \frac{1}{4 \pi \sqrt{1-c^2}} \ln \left ( \frac{y_1^2}{1-c^2} + y_2^2 \right )$, thus
\[ S_c h(z) = - \frac{c^2}{4 \pi \sqrt{1-c^2}} \int_{\R^2} \ln \left ( \frac{(z_1-y_1)^2}{1-c^2} + (z_2-y_2)^2 \right ) \partial_{y_1}^2 h(y_1 \, , y_2) \, \text{d}y . \]
Changing variables, we obtain
\[ S_c h (\nu z_1 \, , z_2 ) = - \frac{c^2}{4 \pi} \int_{\R^2} \ln \left ( |z - \zeta |^2 \right ) \partial_{y_1}^2 h( \nu \zeta_1 \, , \zeta_2 ) \, \text{d} \zeta \]
where $\nu = \sqrt{1-c^2}$. Note that $\nu^{-1} \lesssim 1$. We then split the integral: $S_ch( \nu z_1 \, , z_2 ) = - \frac{c^2}{4 \pi^2} (I_1 + I_2 + I_\star )$ where
\begin{align*}
& I_1 = \int_{\substack{|\zeta| \geqslant |z|^{1/K} \\ |\zeta-z| \leqslant \frac{1}{2} |z|}} \ln \left ( |z - \zeta |^2 \right ) \partial_{y_1}^2 h( \nu \zeta_1 \, , \zeta_2 ) \, \text{d} \zeta , \\
& I_2 = \int_{\substack{|\zeta| \geqslant |z|^{1/K} \\ |\zeta-z| > \frac{1}{2} |z|}} \ln \left ( |z - \zeta |^2 \right ) \partial_{y_1}^2 h( \nu \zeta_1 \, , \zeta_2 ) \, \text{d} \zeta \\
\text{and} \quad & I_\star = \int_{|\zeta|<|z|^{1/K}} \ln \left ( |z - \zeta |^2 \right ) \partial_{y_1}^2 h( \nu \zeta_1 \, , \zeta_2 ) \, \text{d} \zeta .
\end{align*}
We begin with $I_1$. For $|\zeta-z| \leqslant \frac{1}{2} |z|$, $|\zeta| \geqslant \frac{1}{2} |z|$ thus $|\partial_{y_1}^2 h( \nu \zeta_1 \, , \zeta_2 ) | \lesssim e^{- \frac{\nu^2}{2} | \zeta |^2} \lesssim e^{- \frac{\nu^2}{4} |z|}$. Hence
\begin{align}
|I_1| \lesssim & \, e^{- \frac{\nu^2}{4} |z|} \int_{| \zeta - z | \leqslant \frac{1}{2} |z|} \left | \ln \left ( | z - \zeta |^2 \right ) \right | \text{d} \zeta = e^{- \frac{\nu^2}{4} |z|} \int_{| \xi | \leqslant \frac{1}{2}} \left | \ln \left ( |z \xi |^2 \right ) \right | |z|^2 \, \text{d} \xi \nonumber \\
\lesssim & \, |z|^2 e^{- \frac{\nu^2}{4} |z|} \left ( 2 \ln |z| \times \frac{\pi}{4} + \int_{ | \xi | \leqslant \frac{1}{2}} | \ln (| \xi |^2) | \, \text{d} \xi \right ) \lesssim |z|^2 ( \ln |z| + 1 ) e^{- \frac{\nu^2}{4} |z|} \lesssim e^{- \frac{\nu^2}{8} |z|} . \label{estI1}
\end{align}
Now we estimate $I_2$. For $| \zeta - z | > \frac{1}{2} |z| \geqslant 1$,
\[ | \ln ( |z - \zeta |^2 ) | = \ln ( |z - \zeta |^2 ) \leqslant K |z- \zeta |^{1/K} \lesssim K \left ( |z|^{1/K} + | \zeta |^{1/K} \right ) \lesssim K \left ( |z|^{1/K} + 1 + | \zeta | \right ) . \]
Besides, for $| \zeta | \geqslant |z|^{1/K}$, $| \partial_{y_1}^2 h ( \nu \zeta_1 \, , \zeta_2 ) | \lesssim e^{- \frac{\nu^2}{2} | \zeta |} \lesssim e^{- \frac{\nu^2}{4} |z|^{1/K}} e^{- \frac{\nu^2}{4} | \zeta |}$. Hence
\begin{equation} 
|I_2| \lesssim e^{- \frac{\nu^2}{4} |z|^{1/K}} \int_{\R^2} K \left ( |z|^{1/K} + 1 + | \zeta | \right ) e^{- \frac{\nu^2}{4} | \zeta |} \, \text{d} \zeta \lesssim K \left ( |z|^{1/K} + 1 \right ) e^{- \frac{\nu^2}{4} |z|^{1/K}} \lesssim K e^{- \frac{\nu^2}{8} |z|^{1/K}} . \label{estI2} 
\end{equation}
At last, to estimate $I_\star$ we integrate by parts twice in the variable $\zeta_1$. We compute
\begin{align*}
& \partial_{\zeta_1} \left ( \ln ( |z- \zeta |^2 ) \right ) = - \frac{2 (z_1 - \zeta_1 )}{|z - \zeta |^2} \\
\text{and} \quad & \partial_{\zeta_1}^2 \left ( \ln ( |z - \zeta |^2 ) \right ) = \frac{2}{|z - \zeta |^2} - \frac{4 ( z_1 - \zeta_1 )^2}{| z - \zeta |^4} .
\end{align*}
This leads to $I_\star = I_3 + I_4 + I_5$ where
\begin{align*}
& I_3 = \frac{1}{\nu} \int_{| \zeta_2 | < |z|^{1/K}} \left [ \ln \left ( |z - \zeta |^2 \right ) \partial_{y_1} h ( \nu \zeta_1 \, , \zeta_2 ) \right ]_{\zeta_1 = - \sqrt{|z|^{2/K} - \zeta_2^2}}^{\zeta_1 = \sqrt{|z|^{2/K} - \zeta_2^2}} \text{d}\zeta_2 , \\
& I_4 = - \frac{1}{\nu^2} \int_{| \zeta_2 | < |z|^{1/K}} \left [ \frac{-2 (z_1 - \zeta_1 )}{|z - \zeta |^2} \, h( \nu \zeta_1 \, , \zeta_2 ) \right ]_{\zeta_1 = - \sqrt{|z|^{2/K} - \zeta_2^2}}^{\zeta_1 = \sqrt{|z|^{2/K} - \zeta_2^2}} \text{d}\zeta_2 \\
\text{and} \quad & I_5 = \frac{1}{\nu^2} \int_{| \zeta | < |z|^{1/K}} \left ( \frac{2}{|z - \zeta |^2} - \frac{4 ( z_1 - \zeta_1 )^2}{| z - \zeta |^4} \right ) h ( \nu \zeta_1 \, , \zeta_2 ) \, \text{d} \zeta .
\end{align*}
We first estimate $I_3$. For $| \zeta_2 | < |z|^{1/K}$ and $\zeta_1 = \pm \sqrt{|z|^{2/K} - \zeta_2^2}$, we have $| \zeta | = |z|^{1/K} \leqslant \frac{1}{2} |z|$ thus $|z - \zeta | \geqslant \frac{1}{2} |z| \geqslant 1$ thus $| \ln ( |z - \zeta |^2 ) | \lesssim K ( |z|^{1/K} + 1 + | \zeta | ) \lesssim K ( |z|^{1/K} + 1)$. Besides, for such $\zeta$, $| \partial_{y_1}^2 h ( \nu \zeta_1 \, , \zeta_2 ) | \lesssim e^{- \frac{\nu^2}{2} | \zeta |} \lesssim e^{- \frac{\nu^2}{2} |z|^{1/K}}$. Hence
\begin{equation} 
|I_3| \lesssim \frac{1}{\nu} \, e^{- \frac{\nu^2}{2} |z|^{1/K}} \int_{| \zeta_2 | < |z|^{1/K}} K ( |z|^{1/K} + 1 ) \, \text{d} \zeta_2 \lesssim K \left ( |z|^{1/K} +1 \right ) |z|^{1/K} e^{- \frac{\nu^2}{2} |z|^{1/K}} \lesssim K e^{- \frac{\nu^2}{4} |z|^{1/K}}. 
\label{estI3}
\end{equation}
Now we estimate $I_4$. Again, for $| \zeta_2 | < |z|^{1/K}$ and $\zeta_1 = \pm \sqrt{|z|^{2/K} - \zeta_2^2}$, we have $1 \leqslant \frac{1}{2} |z| \leqslant |z - \zeta | \leqslant |z| + | \zeta | \lesssim |z|$ thus
\begin{equation} 
|I_4| \lesssim \frac{1}{\nu^2} \int_{| \zeta_2 | < |z|^{1/K}} \frac{|z|}{|z|^2} \, e^{- \frac{\nu^2}{2} |z|^{1/K}} \text{d} \zeta_2 \lesssim \frac{1}{\nu^2 |z|} \, e^{- \frac{\nu^2}{2} |z|^{1/K}} \times |z|^{1/K} \lesssim e^{- \frac{\nu^2}{2} |z|^{1/K}}. 
\label{estI4}
\end{equation}
At last we deal with $I_5$ which describes the main asymptotic behavior of $N$. For $| \zeta | < |z|^{1/K} \leqslant \frac{|z|}{2}$, write 
\[ \frac{1}{|z - \zeta |^2} = \frac{1}{|z|^2} \sum\limits_{n=0}^K \frac{\left ( 2 z \cdot \zeta - | \zeta |^2 \right )^n}{|z|^{2n}} + \frac{1}{|z- \zeta |^2} \frac{(2z \cdot \zeta - | \zeta |^2 )^{K+1}}{|z|^{2(K+1)}} \]
which leads to
\begin{equation} 
\frac{1}{\nu^2} \int_{| \zeta | < |z|^{1/K}} \frac{2}{|z - \zeta |^2} h( \nu \zeta_1 \, , \zeta_2 ) \, \text{d} \zeta = \sum\limits_{n=0}^K (\mathbf{J}_n^{\text{main}} + \mathbf{J}_n^{\text{queue}}) + \mathbf{R}_K
\label{JKR1}
\end{equation}
where
\begin{align*}
& \mathbf{J}_n^{\text{main}} = \frac{2}{\nu^2 |z|^{2n+2}} \int_{\R^2} (2z \cdot \zeta - | \zeta |^2 )^n h( \nu \zeta_1 \, , \zeta_2 ) \, \text{d} \zeta , \\
& \mathbf{J}_n^{\text{queue}} = - \frac{2}{\nu^2 |z|^{2n+2}} \int_{| \zeta  | \geqslant |z|^{1/K}} (2z \cdot \zeta - | \zeta |^2 )^n h( \nu \zeta_1 \, , \zeta_2 ) \, \text{d} \zeta \\
\text{and} \quad & \mathbf{R}_K = \frac{2}{\nu^2 |z|^{2K+2}} \int_{ | \zeta | < |z|^{1/K}} \frac{(2 z \cdot \zeta - | \zeta |^2 )^{K+1}}{|z - \zeta |^2} \, h( \nu \zeta_1 \, , \zeta_2 ) \, \text{d} \zeta .
\end{align*}
The term $\mathbf{J}_n^{\text{main}}$ is a rational fraction of $z$, whose coefficients are combinations of pseudo-moments of $h = |U|^2$. Now we estimate the queue terms. First, recalling that $\int_0^{+ \infty} s^m e^{- \gamma s} \text{d}s = \frac{m!}{\gamma^{m+1}}$ for $m \in \N$ and $\gamma > 0$, we compute
\begin{align}
| \mathbf{J}_n^{\text{queue}} | \lesssim & \, \frac{1}{|z|^{2n+2}} \int_{ | \zeta | \geqslant |z|^{1/K}} | \zeta |^n (2|z| + | \zeta |)^n e^{- \frac{\nu^2}{2} | \zeta |} \text{d} \zeta \nonumber \\
\lesssim & \, \frac{1}{|z|^{2n+2}} e^{- \frac{\nu^2}{4} |z|^{1/K}} \int_{| \zeta | \geqslant |z|^{1/K}} n4^n ( | \zeta |^n | z |^n + | \zeta |^{2n} ) e^{- \frac{\nu^2}{4} | \zeta |} \text{d} \zeta \nonumber \\
\lesssim & \, \frac{1}{|z|^{2n+2}} e^{- \frac{\nu^2}{4} |z|^{1/K}} \left ( |z|^n \int_{\R^2} | \zeta |^n e^{- \frac{\nu^2}{4} | \zeta |} \text{d} \zeta + \int_{\R^2} | \zeta |^{2n} e^{- \frac{\nu^2}{4} | \zeta |} \text{d} \zeta \right ) \nonumber \\
\lesssim & \, \frac{1}{|z|^{2n+2}} n (2n)! \left ( \frac{8}{\nu^2} \right )^{2n+1} |z|^n e^{- \frac{\nu^2}{4} |z|^{1/K}} \nonumber \\
\lesssim & \, \frac{1}{|z|^{n+2}} (2n+1)! \left ( 1+8 \nu^{-2} \right )^{2n+1} e^{- \frac{\nu^2}{4} |z|^{1/K}} . \label{Jq1}
\end{align}
For the term $\mathbf{R}_K$, we recall that $||h||_{L^\infty} \lesssim 1$; we get
\begin{align}
| \mathbf{R}_K | \lesssim & \, \frac{1}{|z|^{2K+2}} \int_{ | \zeta | < |z|^{1/K}} \frac{1}{\left ( |z| - \frac{|z|}{2} \right )^2} \left ( 2 |z| \times |z|^{1/K} + |z|^{2/K} \right )^{K+1} \, \text{d} \zeta \nonumber \\
\lesssim & \, \frac{1}{|z|^{2K+4}} \left ( 3 |z|^{1+ \frac{1}{K}} \right )^{K+1} |z|^{2/K} \nonumber \\
\lesssim & \, \frac{3^K}{|z|^{K+1}}. \label{R1}
\end{align}
We proceed similarly to estimate the behavior of the second term in $I_5$. For $| \zeta | < |z|^{1/K} \leqslant \frac{|z|}{2}$, write
\[ \frac{1}{| z - \zeta |^4} = \frac{1}{|z|^4} \sum\limits_{n=0}^K \frac{n+1}{|z|^{2n}} \left ( 2z \cdot \zeta - | \zeta |^2 \right )^n + \frac{1}{|z - \zeta |^4} \frac{\left ( 2 z \cdot \zeta - | \zeta |^2 \right )^{K+1} \left ( K+2 - (K+1) \frac{2 z \cdot \zeta - | \zeta |^2}{|z|^2} \right )}{|z - \zeta |^4} \]
which leads to
\begin{equation} 
- \frac{1}{\nu^2} \int_{| \zeta | < |z|^{1/K}} \frac{4(z_1 - \zeta_1)^2}{|z - \zeta |^4} h( \nu \zeta_1 \, , \zeta_2 ) \, \text{d} \zeta = \sum\limits_{n=0}^K (\tilde{\mathbf{J}}_n^{\text{main}} + \tilde{\mathbf{J}}_n^{\text{queue}}) + \tilde{\mathbf{R}}_K
\label{JKR2}
\end{equation}
where
\begin{align*}
& \tilde{\mathbf{J}}_n^{\text{main}} = - \frac{4(n+1)}{\nu^2 |z|^{2n+4}} \int_{\R^2} (z_1 - \zeta_1)^2 (2z \cdot \zeta - | \zeta |^2 )^n h( \nu \zeta_1 \, , \zeta_2 ) \, \text{d} \zeta , \\
& \tilde{\mathbf{J}}_n^{\text{queue}} = \frac{4(n+1)}{\nu^2 |z|^{2n+4}} \int_{| \zeta  | \geqslant |z|^{1/K}} (z_1 - \zeta_1)^2 (2z \cdot \zeta - | \zeta |^2 )^n h( \nu \zeta_1 \, , \zeta_2 ) \, \text{d} \zeta \\
\text{and} \quad & \tilde{\mathbf{R}}_K = - \frac{4}{\nu^2 |z|^{2K+2}} \int_{ | \zeta | < |z|^{1/K}} \frac{(z_1 - \zeta_1)^2 (2 z \cdot \zeta - | \zeta |^2 )^{K+1} \left ( K+2 - (K+1) \frac{2 z \cdot \zeta - | \zeta |^2}{|z|^2} \right )}{|z - \zeta |^4} \, h( \nu \zeta_1 \, , \zeta_2 ) \, \text{d} \zeta .
\end{align*}
Similarly to the obtention of \eqref{Jq1} and \eqref{R1}, we find
\begin{align}
& | \tilde{\mathbf{J}}_n^{\text{queue}} | \lesssim \frac{1}{|z|^{n+2}} (2n+4)! (1+8 \nu^{-2} )^{2n+3} e^{- \frac{\nu^2}{4} |z|^{1/K}} \label{Jq2} \\
\text{and} \quad & | \tilde{\mathbf{R}}_K | \lesssim \frac{K3^K}{|z|^{K+1}} . \label{R2}
\end{align}
Combining \eqref{JKR1}, \eqref{Jq1}, \eqref{R1}, \eqref{JKR2}, \eqref{Jq2} and \eqref{R2}, it holds that
\begin{equation} 
\left | I_5 - \sum\limits_{n=0}^K \left ( \mathbf{J}_n^{\text{main}} + \tilde{\mathbf{J}}_n^{\text{main}} \right ) \right | \lesssim \frac{K3^K}{|z|^{K+1}} + (2K+5)! (1+8 \nu^{-2} )^{2K+3} e^{- \frac{\nu^2}{4} |z|^{1/K}} . 
\label{estI5}
\end{equation}
Gathering \eqref{estI1}, \eqref{estI2}, \eqref{estI3}, \eqref{estI4} and \eqref{estI5}, we find that, for all $y \in \R^2$ such that $|y| \geqslant 4$, setting $z_1 = \frac{y_1}{\nu}$ and $z_2 = y_2$ (note that $| y | \leqslant |z| \leqslant \nu^{-1} |y|$),
\begin{align*}
& \left | S_c h (y_1 \, , y_2) + \frac{c^2}{4 \pi^2} \sum\limits_{n=0}^K ( \mathbf{J}_n^{\text{main}} + \tilde{\mathbf{J}}_n^{\text{main}} ) ( \nu^{-1} y_1 \, , y_2 ) \right | \\
& = \left | S_c h \left ( \nu z_1 \, , z_2 \right ) + \frac{c^2}{4 \pi^2} \sum\limits_{n=0}^K (\mathbf{J}_n^{\text{main}} + \tilde{\mathbf{J}}_n^{\text{main}} ) (z_1 \, , z_2) \right | \\
& \lesssim c^2 \left ( \frac{K3^K}{|z|^{K+1}} + (2K+5)! (1+8 \nu^{-2} )^{2K+3} e^{- \frac{\nu^2}{8} |z|^{1/K}} \right ) \\
& \lesssim \frac{c^2 K 3^K}{|y|^{K+1}} + c^2 (2K+5)! (1+8 \nu^{-2})^{2K+3} e^{- \frac{\nu^2}{8} |y|^{1/K}} .
\end{align*}
Recalling that $N_c = -|U_c|^2 - S_c (|U_c|^2)$ and $|U_c|^2 (y) \lesssim e^{- |y|}$, we obtain the desired expansion \eqref{expanN}. 

In order to get the decrease estimates \eqref{decrN} for $N$ and its derivatives, we follow the same proof as above but we do not need to analyse precisely the contributions of $I_5$. We take $K=3$ for example. We simply write that, for $|z| \geqslant 4$ and $| \zeta  | < |z|^{1/3} \leqslant \frac{|z|}{2}$, we have $|z - \zeta | \geqslant \frac{|z|}{2}$ thus
\[ |I_5| \lesssim  \frac{1}{\nu^2} \int_{| \zeta | < |z|^{1/3}} \left ( \frac{1}{|z|^2} + \frac{|z|^2}{|z|^4} \right ) |h( \nu \zeta_1 \, , \zeta_2 )| \, \text{d} \zeta \lesssim \frac{1}{|z|^2} \int_{\R^2} |h| \lesssim \frac{1}{|z|^2} . \]
This estimate, combined with the previous ones \eqref{estI1}, \eqref{estI2}, \eqref{estI3} and \eqref{estI4}, leads to 
\[ |N_c(y)| \lesssim e^{-|y|} + c^2 \left ( e^{- \frac{\nu^2}{8} |z|^{1/3}} + \frac{1}{|z|^2} \right ) \lesssim e^{-|y|} + \frac{c^2}{|y|^2}. \]
The estimates for higher derivatives are obtained similarly; we write that
\[ \partial^m S_c h = S_c \left ( \partial^m h \right ) . \]
Now we simply need to differentiate the function $\ln \left ( |z - \zeta |^2 \right )$ more times (with regards to $\zeta_1$ and/or $\zeta_2$) and to integrate by parts as many times. For instance, 
\[ \partial_{\zeta_1}^2 \partial_{\zeta_2} \left ( \ln | z - \zeta |^2 \right ) = 4 \frac{z_2- \zeta_2}{|z- \zeta |^4} -16 \frac{(z_1- \zeta_1 )^2 ( z_2 - \zeta_2)}{|z - \zeta |^6} . \]
We can easily see that, for any $m \in \N^2$,
\[ \left | \partial^m \left ( \ln |z - \zeta |^2 \right ) \right | \lesssim_{m} \frac{1}{|z - \zeta |^{|m|+2}} . \]
Following the same steps as above to control the derivatives of $N_c$, we deduce the estimate \eqref{decrN} for any $m \in \N^2$. The estimates on $V$ are obtained similarly since, recalling \eqref{sysNUV2}, $V_c = \left ( \begin{array}{c} - c \Delta_c^{-1} \partial_{y_1}^2 (|U_c|^2) \\ - c \Delta_c^{-1} \partial_{y_1} \partial_{y_2} (|U_c|^2) \end{array} \right )$. 
\end{proof}

\begin{remark} \label{rkexpanV}
Adapting the proof above, we perfectly can compute expansions for the derivatives of $N$, or for $V$ or its derivatives, similar to the expansion \eqref{expanN}. We do not pursue such computations here. 
\end{remark}

\section*{Appendix}
Here we prove the second point of Lemma \ref{lemL}. Let $f_2 \in H_{oe}^2$. We aim to prove that there exists $g_2 = L_-^{-1} f_2$ such that $L_- g_2 = f_2$; and to prove that $||g_2||_{H^2} \lesssim ||f_2||_{L^2}$. We adapt the arguments in \cite[Propositions 4.1 and 4.5, Lemma 4.6]{GM}. Recall, from \cite[Eq. (1.24)]{MaRa} for example, that
\begin{equation} 
\forall w \in H^1 ( \R^2 \, , \R ) \, \, \text{such that} \, \, w \perp \rho , \partial_{y_1} Q , \partial_{y_2} Q , \quad \langle L_- w \, , w \rangle \gtrsim ||w||_{H^1}^2 
\label{coer1}
\end{equation}
where $\langle w_1 \, , w_2 \rangle = \int_{\R^2} w_1 w_2$ and $\rho \in H^2 ( \R^2 )$ is a radial function such that $L_+ \rho = \frac{|y|^2Q}{4}$. The orthogonality assumptions above concern the scalar product $\langle \cdot , \cdot \rangle$. In our situation, $f_2$ is $y_1$-odd and $y_2$-even, thus $f_2 \perp \rho , \partial_{y_2}Q$. However $f_2 \not\perp \partial_{y_1} Q$ \textit{a priori}. We introduce 
\[ \mathbf{M} = \left \{ g \in H^1 ( \R^2 ) \, \, \text{such that $g$ is $y_1$-odd and $y_2$-even, and $g \perp \partial_{y_1}Q$} \right \} . \]
We write $f_2 = f_2^\perp + \alpha \partial_{y_1} Q$ with $\alpha = \frac{\langle f_2 , \partial_{y_1} Q \rangle}{|| \partial_{y_1} Q ||_{L^2}^2}$. Note that $\partial_{y_1} Q \perp \rho , \partial_{y_2} Q$ for the same reasons as $f_2$. Thanks to the choice of $\alpha$, we have $f_2^\perp \in \mathbf{M}$. The expression of $\alpha$ also gives 
\begin{align}
& | \alpha | \lesssim | \langle f_2 \, , \partial_{y_1} Q \rangle | \lesssim ||f_2||_{L^2} \label{alpha} \\
\text{hence} \quad & ||f_2^\perp ||_{L^2} \lesssim ||f_2||_{L^2} + | \alpha | \lesssim ||f_2||_{L^2} . \label{f2perp}
\end{align}
Let $\Pi_\mathbf{M}$ be the orthogonal projection on $\mathbf{M}$ and $L_{- \mathbf{M}} = \Pi_\mathbf{M} \circ L_- \, : \, \mathbf{M} \to \mathbf{M}$. On the Hilbert space $\mathbf{M}$ we use the norm $|| \cdot ||_{H^1}$.
\begin{itemize}
	\item[$\bullet$] The bilinear form $(g^\perp \, , w) \in \mathbf{M} \times \mathbf{M} \mapsto \langle L_{- \mathbf{M}} g^\perp \, , w \rangle \in \R$ is continuous. Indeed, for any $g^\perp , w \in \mathbf{M}$,
	\begin{align*}
	\langle L_{- \mathbf{M}} g^\perp \, , w \rangle =& \, \langle \Pi_{\mathbf{M}} L_- g^\perp \, , w \rangle = \langle L_- g^\perp \, , \Pi_{\mathbf{M}} w \rangle = \langle L_- g^\perp \, , w \rangle \\
	=& \, \langle \nabla g^\perp \, , \nabla w \rangle + \langle (1-Q^2) g^\perp \, , w \rangle \\
	\text{thus} \quad \left | \langle L_{- \mathbf{M}} g^\perp \, , w \rangle \right | \lesssim & \, ||g^\perp ||_{H^1} ||w||_{H^1} .
	\end{align*}
	Moereover, this bilinear form is coercive. Indeed, for any $w \in \mathbf{M}$, it holds that $w \perp \rho , \partial_{y_1} Q , \partial_{y_2} Q$ thus
	\begin{equation}
	\langle L_{- \mathbf{M}} w \, , w \rangle = \langle \Pi_{\mathbf{M}} L_- w \, , w \rangle = \langle L_- w \, , \Pi_{\mathbf{M}} w \rangle = \langle L_- w \, , w \rangle \geqslant ||w||_{H^1}^2 
	\label{coer2}
	\end{equation}
	thanks to \eqref{coer1}. 
	\item[$\bullet$] The linear form $w \in \mathbf{M} \mapsto \langle f_2^\perp \, , w \rangle \in \R$ is continuous. Indeed, for any $w \in \mathbf{M}$, $| \langle f_2^\perp \, , w \rangle | \leqslant ||f_2^\perp||_{L^2} ||w||_{H^1}$.
\end{itemize}
Henceforth, the Lax-Milgram theorem gives the existence of a unique $g^\perp \in \mathbf{M}$ such that
\begin{align*}
& \forall w \in \mathbf{M} , \quad \langle L_{- \mathbf{M}} g^\perp \, , w \rangle = \langle f_2^\perp \, , w \rangle \\
\text{\textit{i.e.}} \quad & L_{- \mathbf{M}} g^\perp = f_2^\perp \\
\text{\textit{i.e.}} \quad & \Pi_{\mathbf{M}} (L_- g^\perp ) = g_2^\perp \\
\text{\textit{i.e.}} \quad & L_- g^\perp = f_2^\perp + \beta \partial_{y_1} Q \quad \text{for some $\beta \in \R$.}
\end{align*}
Using the coercivity property \eqref{coer2} that holds for any function in $\mathbf{M}$ (in particular $g^\perp$), we have
\begin{align}
& ||g^\perp ||_{H^1}^2 \lesssim \langle L_{- \mathbf{M}} g^\perp \, , g^\perp \rangle = \langle f_2^\perp \, , g^\perp \rangle \leqslant ||f_2^\perp ||_{L^2} ||g^\perp||_{H^1} \nonumber \\
\text{thus} \quad & ||g^\perp ||_{H^1} \lessim ||f_2^\perp ||_{L^2} \lesssim ||f_2||_{L^2} \quad \text{thanks to \eqref{f2perp}.} \label{gperp}
\end{align}
Additionally, 
\begin{align}
& f_2^\perp \perp \partial_{y_1} Q \quad \text{thus} \quad \beta = \frac{1}{|| \partial_{y_1} Q ||_{L^2}^2} \langle L_- g^\perp , \partial_{y_1} Q \rangle = \frac{1}{|| \partial_{y_1} Q ||_{L^2}^2} \langle g^\perp , L_- \partial_{y_1} Q \rangle \nonumber \\
\text{thus} \quad & | \beta | \lesssim ||g^\perp ||_{L^2} \lesssim ||f_2||_{L^2} \quad \text{thanks to \eqref{gperp}.} \label{beta}
\end{align}
We now temporarily admit that there exists $R \in H_{oe}^2$ such that $L_- R = \partial_{y_1} Q$ (see Lemma \ref{lemR} below). Set $g = g^\perp + ( \alpha - \beta ) R$. Clearly, $g \in H_{oe}^2$ and
\[ L_- g = (f_2^\perp + \beta \partial_{y_1} Q ) + ( \alpha - \beta ) \partial_{y_1} Q = f_2^\perp + \alpha \partial_{y_1} Q = f_2. \]
Besides, combining \eqref{alpha}, \eqref{gperp} and \eqref{beta}, it holds that
\[ || g ||_{H^1} \lesssim ||g^\perp ||_{H^1} + | \alpha | + | \beta | \lesssim ||f_2||_{L^2} . \]
At last, we control the $H^2$-norm of $g$ as follows:
\[ \Delta g = g - Q^2 g - f_2 \quad \text{thus} \quad || \Delta g ||_{L^2} \lesssim ||g||_{L^2} + ||f_2||_{L^2} \lesssim ||f_2||_{L^2} . \]
Hence $||g||_{H^2} \lesssim ||f_2||_{L^2}$ which is the estimate announced in Lemma \ref{lemL}. \hfill \qedsymbol

\begin{lemma} \label{lemR}
There exists $R \in H_{oe}^2$ such that $L_- R = \partial_{y_1} Q$.
\end{lemma}

\begin{proof}
Recall that $\text{Ker} (L_-) = \text{span} (Q)$. Since $\partial_{y_1} Q \perp Q$, it results that there exists a unique $R \in H^1 ( \R^2 )$ such that $R \perp Q$ and $L_-R= \partial_{y_1} Q$. Since $\Delta R = R-Q^2R- \partial_{y_1}Q \in L^2$, it follows that $R \in H^2 ( \R^2 )$.
\begin{itemize}
	\item[$\bullet$] Set $\tilde{R} (y_1 \, , y_2) = R(y_1 \, , -y_2)$. Since $Q$ and $\partial_{y_1} Q$ are $y_2$-even, it holds that
	\begin{align*}
	& (L_- \tilde{R} ) (y_1 \, , y_2) = (L_- R)(y_1 \, , -y_2) = ( \partial_{y_1} Q)(y_1 \, , -y_2) = ( \partial_{y_1} Q) (y_1 \, , y_2) \\
	\text{and} \quad & \langle \tilde{R} \, , Q \rangle = \langle R \, , Q \rangle = 0.
	\end{align*}
	The uniqueness of $R$ then ensures that $\tilde{R}=R$; thus $R$ is $y_2$-even.
	\item[$\bullet$] Similarly, set $\check{R} (y_1 \, , y_2) = -R(-y_1 \, , y_2)$. Since $Q$ is $y_1$-even while $\partial_{y_1} Q$ is $y_1$-odd, it holds that
	\begin{align*}
	& (L_- \check{R} ) (y_1 \, , y_2) = -(L_- R)(-y_1 \, , y_2) = -( \partial_{y_1} Q)(-y_1 \, , y_2) = ( \partial_{y_1} Q) (y_1 \, , y_2) \\
	\text{and} \quad & \langle \check{R} \, , Q \rangle = - \langle R \, , Q \rangle = 0.
	\end{align*}
	The uniqueness of $R$ then ensures that $\check{R}=R$; thus $R$ is $y_1$-odd.
\end{itemize}
Hence $R \in H_{oe}^2$ as announced.
\end{proof}

\end{document}